\documentclass{amsart}
\usepackage{amsmath, enumerate}
\usepackage{algorithm}
\usepackage{algpseudocode}
\usepackage{amsfonts}
\usepackage{amssymb}
\usepackage{amsthm}
\usepackage{color}
\usepackage{caption}
\usepackage{graphicx}
\usepackage{url}

\newcommand{\Z}{\mathbb{Z}}

\newcommand{\bp}{\begin{problem}}

\newcommand{\ep}{\end{problem}}

\newcommand{\ba}{\begin{answer}}

\newcommand{\ea}{\end{answer}}

\newcommand{\ben}{\renewcommand{\theenumi}{\alph{enumi}}

\renewcommand{\labelenumi}{(\theenumi)}\begin{enumerate}}

\newcommand{\een}{\end{enumerate}}

\newcommand{\Aut}{\mathrm{Aut}}

\newcommand{\Out}{\mathrm{Out}}

\newcommand{\GL}{\mathrm{GL}}
\newcommand{\st}{\mathrm{st}}
\newcommand{\lk}{\mathrm{lk}}

\newtheorem{defin}{Definition}[section]
\newtheorem{thm}[defin]{Theorem}
\newtheorem{cor}[defin]{Corollary}

\newtheorem{lem}[defin]{Lemma}
\newtheorem{prop}[defin]{Proposition}
\newtheorem{conj}[defin]{Conjecture}

\DeclareCaptionType{result}[Table][List of Results]

\title[Cryptography with  groups]{Algorithmic Problems in right-angled Artin groups: Complexity and Applications}
\begin{document}
\date{\today}
\author[R. Flores]{Ram\'{o}n Flores}
\address{Ram\'{o}n Flores, Department of Geometry and Topology, University of Seville, Spain}
\email{ramonjflores@us.es}

\author[D. Kahrobaei]{Delaram Kahrobaei}
\address{Delaram Kahrobaei, Department of Computer Science, University of York, UK, CUNY Graduate Center, New York University, Tandon School of Engineering}
\email{dk2572@nyu.edu}

\author[T. Koberda]{Thomas Koberda}
\address{Thomas Koberda, Mathematics Department, University of Virginia}
\email{thomas.koberda@gmail.com}

\begin{abstract}
In this paper we consider several classical and novel algorithmic problems for right-angled Artin groups, some of which are closely related to graph theoretic problems, and study their computational complexity. We study these problems with a view towards applications to cryptography.
\end{abstract}

\maketitle

\section{Introduction and motivation}
\label{Intro}

In this paper, we investigate various group theoretic problems in right-angled Artin group theory, with a view towards their applications to cryptography via computational complexity. Right-angled Artin groups (which in the literature are sometimes referred to as partially commutative groups, graph groups, or simply RAAGs), are central objects in geometric group theory, and have been studied from different algorithmic perspectives as outlined in some detail in Section \ref{Complexity} below. Because of their tractable algorithmic properties, they have also been proposed as possible platforms for some cryptographic schemes~\cite{WM,BMS,LP}.

Since right-angled Artin groups are uniquely defined by a finite simplicial graph and vice versa, it is clear from the outset that  there is a natural connection between algorithmic graph theoretic problems and group theoretic problems for right-angled Artin groups. Since the graph theoretic problems have been of central importance in complexity theory, it is natural to consider some of these graph theoretic problems via their equivalent formulation as group theoretic problems about right-angled Artin groups. The theme of the paper is to convert graph theoretic problems for finite graphs into group theoretic ones for right-angled Artin groups, and to investigate the graph theory algebraically. In doing so, new approaches to resolving problems in complexity theory become apparent. We are primarily motivated by the fact that some of these group theoretic problems can be used for cryptographic purposes, such as authentication schemes, secret sharing schemes, and key exchange problems. Moreover, efficient presentations of groups such as right-angled Artin groups make certain computations easier and therefore make such classes of groups more suitable for practical applications.

We now give a more detailed description of the contents of this paper. In Section \ref{RAAGs}, we offer a brief survey of right-angled Artin groups from a group-theoretic point view, and give an account of the current status of certain natural algorithmic problems which arise in this context. Section \ref{autoRAAG} contains the primary algebraic result in this paper, which translates between a graph theoretic problem and the group theory of right-angled Artin groups, and which to the knowledge of the authors is new:

\begin{thm}\label{thm:autraag}
Let $\Gamma$ be a finite simplicial graph and let $A(\Gamma)$ be the corresponding right-angled Artin group. Then $\Gamma$ admits a nontrivial automorphism if and only if $\Out(A(\Gamma))$ contains a finite nonabelian subgroup.
\end{thm}

In particular, this result implies the solvability of a certain subgroup problem in the automorphism group of a right-angled Artin group, as is explained in the sequel. In Section \ref{DecProblem} we prove that the decomposition problem for right-angled Artin groups is solvable in polynomial time. Section \ref{Membership} is devoted to the membership problem, where we appeal to $3$--manifold techniques to study surface subgroups of right-angled Artin groups. Applications to cryptography are analyzed in Section \ref{crypto}, which contains a review of the current status of right-angled Artin groups in cryptography, a new sharing scheme based on the decomposition problem, and a prospective discussion of these groups as a platform for an efficient key exchange based on the membership problem for distorted surface subgroups. In Section \ref{open}, we translate some further classical algorithmic graph problems into the language of right-angled Artin groups, and we conclude with a list of open questions in the framework of cryptography and right-angled Artin groups.

\section{Background on right-angled Artin groups and complexity theory}
\label{RAAGs}
\subsection{Basic definitions and facts}

The class of groups which is our primary interest in this paper is the class of right-angled Artin groups, a class which appears to have been introduced by Hauschild and Rautenberg in~\cite{HR71} (though the closely related notion of a partially commutative monoid was
studied somewhat earlier by Cartier and Foata~\cite{CF69}).
These groups were then called \emph{semifree groups}, a term which has since fallen out of fashion. Detailed surveys about the structure and applications of these groups can be found in \cite{koberda} and \cite{charney}, whereas a general introduction to combinatorial group theory which is broadly applicable in this context can be found in~\cite{MKS}.

\begin{defin} [Right-angled Artin groups] Let $\Gamma$ be a finite simplicial graph. We write $V = V (\Gamma)$ for the finite set of vertices and $E(\Gamma) \subset V \times V$ for the set of edges, viewed as unordered pairs of vertices.  The \emph{right-angled Artin group} on $\Gamma$ is the group
$$A(\Gamma) = \langle V|[v_i, v_j] =1 \text{ whenever } (v_i, v_j) \in E \rangle.$$
In other words, $A(\Gamma)$ is generated by the vertices of $\Gamma$, and the only relations are given by commutation of adjacent vertices.
\end{defin}

The requirement that $\Gamma$ be simplicial in the definition above simply means that the diagonal of $V \times V$ is excluded from the set of edges, and only one edge is allowed between any pair of vertices. It is standard to refer to the generators $V$ of $A(\Gamma)$ with the presentation given above as \emph{vertex generators} of $A(\Gamma)$. The defining graph $\Gamma$ is an isomorphism invariant for $A(\Gamma)$ (~\cite{DromsProc87}, cf.~\cite{KMNR,KR80}), though vertex generators of a right-angled Artin group are generally not canonical.

\subsection{Complexity of the algorithmic problems for right-angled Artin groups}
\label{Complexity}
In this section we summarize the status of the complexity of some algorithmic problems in the context of right-angled Artin groups. We highlight the particular easy and hard group theoretical problems which are relevant to each point, some of which will be relevant in the sequel.

\begin{itemize}
\item The {\bf Word Problem} for right-angled Artin groups was proved to be solvable in linear time by Liu--Wrathall--Zeger~\cite{LWZ}, and the result is also true for free partially commutative monoids. Crucial use of the Viennot piling \cite{Viennot} is used in establishing these complexity results. An excellent exposition and bibliography on this topic can be found in~\cite{CGW}. A natural variation on the word problem is the
{\bf Word Choice Problem}. In this problem, one considers fixed words $a$ and $b$ in a fixed finite generating set for a group $G$, and then one takes
a third word $c$ which is known to be equal to one of $a$ or $b$. The word choice problem is to decide if $c$ is equal to $a$. It is
straightforward to see that the word choice problem for right-angled Artin groups is solvable in at most linear time (in the lengths of $a,b,c$),
since the word problem is already solvable in linear time.

\item In the same paper \cite{CGW}, Crisp--Godelle--Wiest adapt the piling method to show that the {\bf Conjugacy Problem} for right-angled Artin groups can also be solved in linear time. Moreover, they prove that the conjugacy problem remains solvable in linear time for certain distinguished subgroups of right-angled Artin groups, such as graph braid groups and certain families of word-hyperbolic and surface groups.

\item It is proved in \cite{DromsProc87} that two right-angled Artin groups are isomorphic if and only if their associated graphs are isomorphic. It follows that the {\bf Group Isomorphism Problem} for right-angled Artin groups is equivalent to the graph isomorphism problem. By a recent celebrated result of Babai~\cite{Babai15} (cf.~\cite{HBD}), the latter is solvable in quasi-polynomial time. On the other hand, Bridson~\cite{MB} has proved that there exist right-angled Artin groups such that the Isomorphism Problem is unsolvable for the class of their finitely presented subgroups. These subgroups are of course necessarily not right-angled Artin groups themselves.

\item One can formulate a restricted version of the {\bf Group Homomorphism Problem} for right-angled Artin groups, where one insists that choices for Artin generators of the source and target groups be made and that homomorphisms of groups take generators to generators.
To avoid certain trivialities, one insists furthermore that group homomorphisms take distinct commuting Artin generators to distinct commuting Artin generators.
As in the previous item, this restricted group homomorphism problem for right-angled Artin groups is equivalent to the graph homomorphism problem, as formulated in~\cite{GJ}, pages 202--203. By colorability considerations, this problem can be seen to be NP--complete whenever the image graph contains triangles (cf.~\cite{Levin}).

\item Using fundamental groups of special cube complexes (see~\cite{HaglundWise} for a definition) which are finite index subgroups of word-hyperbolic groups obtained by the Rips construction over a free group, Bridson~\cite{BM} proves that there are classes of right-angled Artin groups in which the {\bf Subgroup Isomorphism Problem} is unsolvable. The reader is also directed to~\cite{C} for a different approach to these kind of problems.

\item Bridson also constructs examples of finitely presented subgroups of right-angled Artin groups for which both the conjugacy problem and the {\bf Membership Problem} are unsolvable. They are subgroups of products of virtually special hyperbolic groups that project onto groups of type $F_3$ with unsolvable word problem. See~\cite{MB}, section 3.

\item It was proved in \cite{DKL} that computing the shortlex normal form in right-angled Artin groups can be done in polynomial time (in fact quadratic; see also \cite{HR}). As these forms are geodesic words, the {\bf Geodesic Length Problem} has polynomial complexity for these groups, as well as the {\bf Geodesic Problem} and the {\bf Bounded Geodesic Length Problem}, as these three problems can be reduced to each other in polynomial time. A good exposition of these problems can be found in (\cite{MSU2}), section 18.3.

\end{itemize}

\section{Automorphisms of graphs and right-angled Artin groups}
\label{AutProblem}

In this section, we investigate the relationship between the automorphism problem for finite simplicial graphs and automorphism groups of right-angled Artin groups. Recall that the automorphism problem for a graph $\Gamma$ (respectively for a right-angled Artin group $A(\Gamma)$) is to find a nontrivial automorphism of $\Gamma$ (respectively $A(\Gamma)$). We will prove Theorem~\ref{thm:autraag}, which we restate here for the convenience of the reader, and Corollary \ref{autoRAAG} will give the desired connection. In Section \ref{GraphGroups} we will describe the relation between different problems for graphs and right-angled Artin groups.

\begin{thm}\label{thm:raag aut}
Let $\Gamma$ be a finite simplicial graph and let $A(\Gamma)$ be the right-angled Artin group on $\Gamma$. Then the graph $\Gamma$ admits a nontrivial automorphism if and only if there exists a finite nonabelian subgroup of $\Out(A(\Gamma))$.
\end{thm}

We remark that this result holds with the outer automorphism group replaced by the automorphism group of $A(\Gamma)$, but the statement as it is given above is stronger since for a right-angled Artin group $A(\Gamma)$, any finite subgroup of $\Aut(A(\Gamma))$ survives in the quotient $\Out(A(\Gamma))$ (see Lemma~\ref{lem:toinet} below).

To prove Theorem~\ref{thm:raag aut}, we gather some preliminary facts. If $n=|V(\Gamma)|$, there is a natural map $\phi\colon \Aut(A(\Gamma))\to \GL_n(\Z)$ given by the action of the automorphisms of $A(\Gamma)$ on the abelianization of $A(\Gamma)$. The following is a result of E. Toinet~\cite{Toinet}:

\begin{lem}\label{lem:toinet}
Let $I(\Gamma)<\Aut(A(\Gamma))$ be the subgroup inducing the identity on the abelianization of $A(\Gamma)$. Then $I(\Gamma)$ is torsion--free.
\end{lem}

Lemma~\ref{lem:toinet} implies that if $F<\Aut(A(\Gamma))$ is finite then $\phi$ maps $F$ isomorphically onto its image.

A \emph{transvection} is a map $\tau_{v,w}\colon V(\Gamma)\to A(\Gamma)$ which for vertices $v,w\in V(\Gamma)$ sends $v\mapsto vw$ (or $v\mapsto wv$) and which is the identity on the remaining vertices of $\Gamma$. A vertex $w$ \emph{dominates} a vertex $v$ if $\lk(v)\subset\st(w)$. It is not difficult to check that if $w$ dominates $v$ then the corresponding transvection (a \emph{dominated transvection}) extends to an automorphism of $A(\Gamma)$. It is not entirely trivial to show but it is true that dominated transvection is a transitive relation:

\begin{lem}[\cite{KMNR}; see also~\cite{Servatius89} and~\cite{CV2009}, Lemma 2.2]\label{lem:transitive}
The relation of dominated transvection is transitive.
\end{lem}

If $\st(v)$ separates $\Gamma$ into components $\Gamma_1,\ldots,\Gamma_k$ for some vertex $v\in V(\Gamma)$, then the map $u\mapsto vuv^{-1}$ for $u\in V(\Gamma_1)$ and which fixes the other vertices of $\Gamma$ extends to an automorphism of $A(\Gamma)$, called a \emph{partial conjugation}.

The following result was conjectured by Servatius~\cite{Servatius89} and proved by Laurence~\cite{Laurence}:

\begin{thm}\label{thm:laurence}
The group $\Aut(A(\Gamma))$ is generated by:
\begin{enumerate}
\item
Graph automorphisms of $\Gamma$;
\item
Vertex inversions $v\mapsto v^{-1}$ for $v\in V(\Gamma)$;
\item
Dominated transvections;
\item
Partial conjugations.
\end{enumerate}
\end{thm}

The following lemma proves the easy direction of Theorem~\ref{thm:raag aut}:

\begin{lem}\label{lem:graph aut}
Suppose $\Gamma$ admits a nontrivial graph automorphism. Then $\Aut(A(\Gamma))$ (and therefore $\Out(A(\Gamma))$ contains a nonabelian finite subgroup.
\end{lem}
\begin{proof}
Let $1\neq Q=\Aut(\Gamma)$ be the group of graph automorphisms of $\Gamma$. Then by Theorem~\ref{thm:laurence}, the group $\Aut(A(\Gamma))$ contains the wreath product $Q\wr(\Z/2\Z)$, which is a semidirect product of $Q$ with $(\Z/2\Z)^n$. Here, each copy of $\Z/2\Z$ is identified with the inversion of some vertex. It is straightforward to see that this wreath product is not abelian.
\end{proof}

We will write $v\geq w$ if the vertex $v$ dominates the vertex $w$. By Lemma~\ref{lem:transitive}, $x\geq v\geq w$ implies $x\geq w$. Let $\{v_0,\ldots,v_{k-1}\}\subset V(\Gamma)$ be distinct vertices. We say that these vertices form a \emph{domination chain} if \[v_{k-1}\geq v_{k-2}\geq\cdots v_1\geq v_0.\] We say that these vertices form a \emph{domination cycle} if \[v_0\geq v_{k-1}\geq v_{k-2}\geq\cdots v_1\geq v_0.\] The following lemma shows that domination cycles give rise to graph automorphisms:

\begin{lem}\label{lem:dom cycle}
Suppose $k\geq 2$ and that $\{v_0,\ldots,v_{k-1}\}\subset V(\Gamma)$ forms a domination cycle. Then $\Gamma$ admits a nontrivial graph automorphism.
\end{lem}
\begin{proof}
Note that since the domination relation is transitive, we have that if \[\{v_0,\ldots,v_{k-1}\}\subset V(\Gamma)\] forms a domination cycle then $v_i\geq v_j$ for all $i\neq j$. Choose a pair of vertices in the domination cycle, say $v_0$ and $v_1$. By the definition of domination, we see that if $x\in V(\Gamma)\setminus\{v_0,v_1\}$ then $x$ is adjacent to $v_0$ if and only if it is adjacent to $v_1$. It is clear then that exchanging $v_0$ and $v_1$ induces a nontrivial automorphism of $\Gamma$.
\end{proof}

Lemma~\ref{lem:dom cycle} implies that if $\Gamma$ admits no nontrivial graph automorphisms then we can order the vertices of $\Gamma$ in such a way as to respect the domination relation. This is the basic idea behind proving Theorem~\ref{thm:raag aut}:
\begin{proof}[Proof of Theorem~\ref{thm:raag aut}]
Let $F<\Out(A(\Gamma))$ be a finite subgroup. We have observed that under the map $\phi\colon \Out(A(\Gamma))\to\GL_n(\Z)$, the group $F$ is sent isomorphically to its image.

Note that since partial conjugations preserve the conjugacy class of generators in $A(\Gamma)$, they all lie in the kernel of $\phi$. It follows that the image of $\phi$ is generated by the image of vertex inversions, graph automorphisms, and dominated transvections.

So, we suppose for the remainder of the proof that $\Gamma$ admits no nontrivial graph automorphisms. If no vertex of $\Gamma$ dominates any other vertex, then $\phi(\Out(A(\Gamma)))$ is isomorphic to $(\Z/2\Z)^n$. Therefore, if $F<\Out(A(\Gamma))$ is a nonabelian finite subgroup, there is at least one pair of vertices, one dominating the other.

Since $\Gamma$ admits no nontrivial graph automorphisms, Lemma~\ref{lem:dom cycle} implies that $\Gamma$ contains no domination cycles. We build a directed graph $\Lambda(\Gamma)$ whose vertices are $V(\Gamma)$, and such that $(v,w)$ is a directed edge from $v$ to $w$ if $w\geq v$. An equivalent formulation of Lemma~\ref{lem:dom cycle} is that if $\Gamma$ admits no graph automorphisms then $\Lambda(\Gamma)$ has no cycles.

We define a height function $\lambda$ on $\Lambda(\Gamma)$ as follows. We say that $\lambda(v)=0$ if $v$ does not dominate any other vertices of $\Gamma$. Such vertices always exist, since $\Lambda$ has no cycles. For $n\geq 1$, we say that $\lambda(v)\geq n$ if $v$ dominates a vertex $w$ with $\lambda(w)=n-1$. Then, we define \[\lambda(v)=\min\{n\mid\lambda(v)\geq n\}.\]

Now choose an arbitrary order $\{v_1,\ldots,v_n\}$ on the indices of $V(\Gamma)$ such that if $i<j$ then $\lambda(v_i)\geq \lambda(v_j)$. Abelianizing $A(\Gamma)$, we have that the images of $\{v_1,\ldots,v_n\}$ form a basis for $A(\Gamma)^{ab}\cong\Z^n$, and the transvection $v_i\mapsto v_iv_j$ becomes $v_i\mapsto v_i+v_j$. Note furthermore that the transvection $v_i\mapsto v_iv_j$ exists in $\Aut(A(\Gamma))$ only if $i<j$. It follows that the image of the transvections in $\Aut(A(\Gamma))$ under $\phi$ in $\GL_n(\Z)$ are all simultaneously upper triangular and unipotent.

The image of the vertex inversions in $\GL_n(\Z)$ consists of matrices which are $\pm1$ along the diagonal, and thus generate an abelian subgroup $T<\GL_n(\Z)$. It follows then that the image of $\Aut(A(\Gamma))$ in $\GL_n(\Z)$ is contained in the group of upper triangular matrices with $\pm1$ along the diagonal, which is isomorphic to a semidirect product of $T$ and $U$, where $U$ is the group of unipotent matrices in $\GL_n(\Z)$. Notice that $U$ is torsion--free, so that any finite subgroup of $\GL_n(\Z)$ intersects $U$ trivially. It follows that $\phi(\Aut(A(\Gamma)))$ contains no nonabelian finite subgroups, whence the group $\Out(A(\Gamma))$ contains no nonabelian finite subgroups.
\end{proof}

The following is an immediate consequence of Theorem~\ref{thm:raag aut}, since the automorphism problem for a finite simplicial graph reduces to the graph isomorphism problem and is therefore solvable in quasi-polynomial time~\cite{S17}:
\begin{cor}
\label{autoRAAG}
The problem of finding a finite nonabelian subgroup of the group of (outer) automorphisms of a right-angled Artin group is solvable in quasi-polynomial time.
\end{cor}

\section{The Decomposition Problem}
\label{DecProblem}
In this section we propose a problem which has both graph theoretic and group theoretic analogues and which is efficiently solvable.

\subsection{The complexity of the decomposition problem}
Given a graph $\Gamma$, the \textbf{Decomposition Problem} consists in decomposing $\Gamma$ as a join of simpler graphs, where here simpler means that the join factors have fewer vertices than $\Gamma$. Recall that a graph $J$ is a \emph{join} of two subgraphs $J\cong J_1*J_2$ if $V(J)=V(J_1)\cup V(J_2)$ and if for each pair $v\in V(J_1)$ and $w\in V(J_2)$, we have $\{v,w\}\in E(J)$. A join $J_1*J_2$ is called nontrivial if both $J_1$ and $J_2$ are nonempty. A join decomposition $J=J_1*\cdots *J_n$ of a graph $J$ is maximal if for each $i$, the graph $J_i$ does not decompose as a nontrivial join. A maximal join decomposition of a finite simplicial graph is unique; see the proof of Proposition~\ref{prop:dec} below.

On the algebraic side, the Decomposition Problem for groups seeks a decomposition of a group as a direct product of proper subgroups. For right-angled Artin groups, it is true (though not entirely trivial) that the decomposition problem for the underlying graph and for the group are equivalent. The following result follows from Servatius' Centralizer Theorem (see~\cite{BehrChar2012,Servatius89,koberda}):

\begin{thm}
Let $\Gamma$ be a finite simplicial graph and let $A(\Gamma)$ be the corresponding right-angled Artin group. The group $A(\Gamma)$ decomposes as a nontrivial direct product if and only if $\Gamma$ decomposes as a nontrivial join.
\end{thm}

We propose next an algorithm that, in the context of right-angled Artin groups, solves both problems in polynomial time:

\begin{prop}\label{prop:dec}
Let $\Gamma$ be a finite simplicial graph. Then there is an algorithm which takes as an input $\Gamma$ (namely, a finite set $V=V(\Gamma)$ of vertices and a subset $E=E(\Gamma)$ of unordered pairs of vertices of $V$) and outputs a list of $\{\Gamma_1,\ldots,\Gamma_n\}$ of subgraphs of $\Gamma$ such that $\Gamma$ is isomorphic to the join $\Gamma\cong \Gamma_1*\cdots *\Gamma_n$ and such that each $\Gamma_i$ does not decompose further as a join. This algorithm is polynomial time in $|V(\Gamma)|$.
\end{prop}
\begin{proof}
We first replace $\Gamma$ by its complement graph $X$, which is to say $\{v,w\}\in E(X)$ if and only if $\{v,w\}\notin E(\Gamma)$. Since the number of edges of the complete graph on $V$ has $O(|V|^2)$ vertices, replacing $\Gamma$ by its complement requires only polynomially many computations. The purpose of this step is that the connected components of the graph $X$ are in bijective to maximal join factors of $\Gamma$.

Next, choose an arbitrary ordering on $V$ and sort the edges of $X$ lexicographically, so that if $\{v, w\}\in E$ then $v<w$. This step requires only polynomially many computations, since many sorting algorithms are efficient.

Next, we take $|V|$ urns and place each vertex of $X$ into one of the urns. We then process the list of edges of $X$ so that if $\{v,w\}\in E(X)$ then we combine the urns containing $v$ and $w$ into one urn. When two urns are combined, the vertices in that urn are listed in lexicographical order. After at most $|E(X)|$ steps, we have processed the entire list $E(X)$. Two vertices of $X$ are connected by a path in $X$ if and only if they lie in the same urn, as is clear from the construction.

It follows that the connected components of $X$ are in bijective correspondence with the urns at the end of the process described in the previous paragraph. The vertices lying in a particular urn span a factor in the maximal join decomposition of $\Gamma$.
\end{proof}

We note that Proposition~\ref{prop:dec} is probably well known to graph theorists, and we do not claim originality here. Much investigation of the complexity of various decomposition problems has been carried out by other authors (see~\cite{Holyer,CohenTarsi,LoncPszczola}).

\section{The membership  problem for distorted subgroups of right-angled Artin groups}
\label{Membership}
Recall that given a group $G$ and a subgroup $H<G$, the \textbf{membership problem} consists in deciding if an element of $G$ belongs to $H$. In group-based cryptography, it is useful to produce a finitely generated group $G$ together with a finitely generated subgroup (or oftentimes many subgroups) $H<G$ such that $H$ is exponentially distorted inside of $G$, but so that the membership problem for $H$ in $G$ is efficiently solvable. The theoretical advantage of such a pair $(G,H)$ is that computations for $H$ done inside of $G$ are much quicker by virtue of exponential distortion, but if $H$ is not known to an eavesdropper then computing data about elements of $H$ such as word lengths is prohibitive. Explicit cryptosystems with free-by-cyclic groups as a platform were produced in~\cite{CKL}. We would like to propose platform groups which fit into the preceding discussion of right-angled Artin groups more naturally.

This perspective in mind, we prove the following fact:

\begin{prop}\label{prop:distorted}
There exists a right-angled Artin group $A(\Gamma)$ and a (possibly punctured) surface subgroup $\pi_1(S)$ which is exponentially distorted inside of $A(\Gamma)$, and the membership problem for $\pi_1(S)<A(\Gamma)$ is solvable in at worst exponential time.
\end{prop}

The exact complexity of the membership problem is not clear to us, but we suspect it is at most linear.

Proposition~\ref{prop:distorted} follows from several deep results of other authors, together with some fairly straightforward facts. Recall that a \emph{quasi--isometric embedding} between two finitely generated groups $H$ and $G$ is a function $f\colon H\to G$ for which there is a constant $C>0$ such that the word metric on $H$ and $f(H)$ (as induced from $G$ in the latter case) are $C$--bi-Lipschitz, up to a $C$--additive error. If $H<G$ is a subgroup for which the inclusion map is a quasi--isometric embedding, then we say that $H$ is \emph{undistorted} in $G$.

The following is a result of I. Agol:

\begin{thm}[Theorem 1.1 of~\cite{Agol}]\label{thm:agol}
Let $M$ be a finite volume hyperbolic $3$--manifold, and let $G=\pi_1(M)$. Then there exists a right-angled Artin group $A(\Gamma)$ and a finite index subgroup $G'<G$ such that $G'\to A(\Gamma)$ is an undistorted subgroup.
\end{thm}

The following is a well--known fact due to Thurston:

\begin{thm}[See~\cite{Travaux}]\label{thm:surface bundle}
Let $S$ be an orientable surface of negative Euler characteristic and let $\psi$ be a pseudo-Anosov mapping class of $S$. Then the mapping torus $M=M_{\psi}$ of $\psi$ is a hyperbolic $3$--manifold of finite volume. Moreover, the inclusion map $\pi_1(S)\to \pi_1(M)$ is exponentially distorted.
\end{thm}

The following fact is very easy:

\begin{prop}\label{prop:membership}
Let $\pi_1(S)<\pi_1(M)$ be a fiber subgroup of a hyperbolic $3$--manifold of finite volume. Then the membership problem for $\pi_1(S)$ is solvable in linear time.
\end{prop}
\begin{proof}
We have that $\pi_1(M)$ and $\pi_1(S)$ fit together in an exact sequence of the form \[1\to\pi_1(S)\to\pi_1(M)\to\Z\to 1,\] so that $g\in\pi_1(M)$ lies in $\pi_1(S)$ if and only if $g$ lies in the kernel of a certain homomorphism to $\Z$. If $\pi_1(M)$ is presented as a semidirect product of this form, the membership problem for $\pi_1(S)$ as a subgroup of $\pi_1(M)$ is clearly solvable in linear time, namely by counting the exponent sum of the stable letter of the semidirect product.
\end{proof}

Finally, we need the following general fact:

\begin{lem}\label{lem:memb-general}
Let $G$ be a group with a solvable word problem and let $H$ be a finitely generated undistorted subgroup. Then the membership problem for $H$ in $G$ is solvable. If the word problem in $G$ is solvable in at most exponential time then the membership problem for $H$ is solvable in at most exponential time.
\end{lem}
\begin{proof}
Let $S_H$ and $S_G$ be finite generating sets for $H$ and $G$ respectively, where we may assume without loss of generality that $S_H\subset S_G$. Then there is a constant $C>0$ such that if $h\in H$ has length $n$ with respect to the word metric defined by $S_G$ then the length of $h$ with respect to the word metric defined by $S_H$ is at most $C\cdot n+C$.

Let $g\in G$ be given as a word of length $N$ in $S_G$. We check whether $g$ is equal to some word of length at most $C\cdot N+C$ with respect to the generating set $S_H$. Since the growth rates of $G$ and $H$ are both at most exponential, if the word problem in $G$ is solvable in at most exponential time, we can determine if $g$ is equal to an element of $H$ in a time which is at most an exponential function of $N$.
\end{proof}

We thus obtain the following corollary of the preceding discussion, which clearly implies Proposition~\ref{prop:distorted}:

\begin{prop}\label{prop:membership-detailed}
There exists a right-angled Artin group $A(\Gamma)$ containing a finitely generated free subgroup $F$ or a closed surface subgroup $\pi_1(S)$ which is exponentially distorted in $A(\Gamma)$, and such that the membership problem for this subgroup is solvable in at most exponential time.
\end{prop}

We remark that one can quite easily improve Proposition~\ref{prop:membership-detailed} to allow for infinitely many pairwise distinct free subgroups and surface groups, by an easy application of the Thurston norm~\cite{ThurstonMemoirs}.

\begin{proof}[Proof of Proposition~\ref{prop:membership-detailed}]
Let $M$ be a finite volume hyperbolic $3$--manifold. By Agol's proof of the Virtual Fibered Conjecture~\cite{Agol}, there is a finite cover $M'$ of $M$ which fibers over the circle, and such that $\pi_1(M')$ embeds quasi--isometrically into a right-angled Artin group $A(\Gamma)$ (cf. Theorem~\ref{thm:agol}). The fiber subgroup of $\pi_1(M')$ will be a closed surface group $\pi_1(S)$ or a finitely generated free group $F$, depending on whether $M$ is closed or not. Since the group $\pi_1(M')$ is undistorted in $A(\Gamma)$ and since the fiber subgroup is exponentially distorted in $\pi_1(M')$, the fiber subgroup is exponentially distorted in $A(\Gamma)$.

The group $\pi_1(M')$ is equipped with a homomorphism $\phi$ to $\Z$ for which the kernel is exactly the fiber subgroup. If $\{g_1,\ldots,g_k\}$ are generators for $\pi_1(M')$, the homomorphism $\phi$ is determined by an assignment of an integer to each $g_i$. To solve the membership problem for the kernel subgroup, we first apply Lemma~\ref{lem:memb-general} to determine if a given element of $A(\Gamma)$ lies in $\pi_1(M')$ and express it in terms of the generators $\{g_1,\ldots,g_k\}$. Here, we use the fact that the word problem in a right-angled Artin group is solvable in linear time.

Given an element of $g\in \pi_1(M')$ as a word in $\{g_1,\ldots,g_k\}$, we determine whether or not it lies in the fiber subgroup by adding up the values of $\phi$ on the generators occurring in an expression for $g$ in terms of $\{g_1,\ldots,g_k\}$. This latter membership problem is clearly at most linear in the length of $g$ with respect to $\{g_1,\ldots,g_k\}$. It follows then that the membership problem for the fiber subgroup is solvable in at most exponential time.
\end{proof}

\begin{conj}\label{conj:linear}
Let $A(\Gamma)$ be as in Proposition~\ref{prop:membership-detailed}. Then the membership problem for the corresponding fiber subgroup is solvable in linear time.
\end{conj}

Conjecture~\ref{conj:linear} seems highly plausible, though the details are likely to be quite involved.
We remark that (as we alluded above,) a positive resolution of Conjecture
~\ref{conj:linear} would provide a new platform for the cryptosystems proposed in~\cite{CKL}, especially the secure version of protocol I
therein. The right-angled Artin group would be a public group, and the fiber subgroup would be secret. An eavesdropper would be hindered
greatly by the fact that the fiber subgroup is exponentially distorted in trying to discover any message that maybe transmitted between
the parties Alice and Bob, whereas Bob would be able to efficiently check if a group element sent to him by Alice lies in the fiber subgroup.

\section{Right-angled Artin groups and Cryptography}
\label{crypto}

In this section we survey some applications of the previous discussion to cryptography.

\subsection{Early cryptosystems using right-angled Artin groups and partially commutative monoids}

The first proto-cryptosystem based on groups was proposed by Wagner-Magyarik in \cite{WM}. It was based on a group $\langle X\textrm{ } | \textrm{ }R\rangle$ for which the word choice problem was hard, but such that after adding some further set of relations $S$ to $R$, the word choice problem becomes easy. In their formulation, they proposed $S$ to be such that $\langle X\textrm{ } | \textrm{ }R\cup S\rangle$ was a right-angled Artin group, as both the word problem and the word choice problem in these groups is linear. Later, Birget-Magliveras-Sramka \cite{BMS} proposed what is sometimes considered to be the first true group-based cryptosystem. Their system is based on a group closely related to the Higman-Thompson groups, though the approach also works for the Higman-Thompson group $G_{3,1}$ and for finite symmetric groups. This cryptosystem makes use of neither an additional set of generators $S$, nor of right-angled Artin groups. The original scheme of Wagner-Magyarik was turned into a true cryptosystem by Levy-Perret \cite{LP}, but using partially commutative monoids instead of groups as a platform. This new scheme was not vulnerable to a reaction attack, which the original proto-cryptosystem was.

\subsection{Authentication schemes}
Based on the work \cite{Grigoriev-Shpilrain}, the first two authors proposed authentication schemes using right-angled Artin groups as a platform~\cite{FK16}. In particular, the two authentication schemes were respectively based on the group homomorphism problem and the subgroup isomorphism problem.

\subsection{Secret sharing schemes}
Many different secret sharing schemes can be constructed using right-angled Artin groups as a platform. Here we describe some representatives.

\subsubsection{The graph decomposition problem}
\label{CryptoDec}
One can build many cryptoschemes which exploit efficient solutions to the decomposition problem, as outlined in Section~\ref{DecProblem}. Since the decomposition problem is efficiently solvable, it is suitable for secret sharing schemes. Here, we record a very simple such scheme, which could serve as the core of a more elaborate secret sharing scheme. One of the novelties of this scheme is that the object being transmitted is a graph and therefore has an intrinsic geometric structure, as opposed to being merely an algebraic datum such as an integer or vector.

As for some technical details, a finite graph can be presented to a computer as a finite list of vertices, and then a finite list of pairs of vertices or as adjacency matrices. The transmission in the scheme below should be done over a secure channel, though more sophisticated versions could avoid this limitation.

The dealer $D$ distributes to each of $n$ participants $P_i\in \{P_1,\ldots,P_n\}$ a connected finite simplicial graph $\Gamma_i$ over secure channel. Each participant computes the number $m_i$ of join factors in a maximal join decomposition of $\Gamma_i$. The participant $P_i$ now knows the value of an unknown monic polynomial $f$ of degree exactly $n$, which satisfies $f(i)=m_i$. The secret is the value $f(0)$.

To make the secret sharing scheme above more algebraic, the dealer may instead deal a right-angled Artin group $A(\Gamma_i)$. The participant $P_i$ computes a bit $b_i$, which takes on the value $0$ if the right-angled Artin group decomposes as a nontrivial direct product, and $1$ otherwise. The unknown polynomial $f$ satisfies $f(i)=b_i$, and the secret is $f(0)$.

The threshold of the previous two secret sharing schemes can easily be changed by tampering with the degree of the monic polynomial $n$, using similar ideas of Shamir's scheme \cite{Shamir} and Lagrange's interpolation.

\subsubsection{The Word Problem}
In \cite{HKS}, Habeeb, Shpilrain and the second author have proposed a secret sharing scheme using the word problem and group presentation.  In \cite{FK16}, the first and second authors proposed right-angled Artin groups as platform for this secret sharing scheme, which is feasible since the word problem in such groups is in linear time.

\subsection{Symmetric key exchanges}
\label{IndiraDel}
In~\cite{CKL}, Chatterji, Lu and the second author proposed several cryptosystems which are based on groups whose geodesic and subgroup membership problems are solvable in polynomial time. This cryptosystem can be modified in a completely straightforward way to use right-angled Artin groups as a platform. The relevant decision problems would come from Proposition~\ref{prop:membership-detailed}, and an efficient cryptosystem (i.e. one which can be implemented in polynomial time) can be constructed assuming Conjecture~\ref{conj:linear}.

\subsection{Connections to braid groups}
In~\cite{HKMPPQ}, a practical cryptanalysis of WalnutDSA was proposed, a platform which was given in 2016 by~\cite{AAGG} as a post-quantum cryptosystem using braid groups and conjugacy search problem. Right-angled Artin groups and braid groups are intimately connected. For instance, Kim and Koberda proved that every right-angled Artin group is a subgroup of some braid group in a way which is undistorted, and hence such an embedding is advantageous with respect to the preservation of algorithmic properties~\cite{KK15}.

\section{Open Problems}
\label{open}
In this last section we will discuss some perspectives of future work for right-angled Artin groups. First we will translate some classical algorithmic problems in graph theory to the groups, and then we will state some cryptography questions in this context.

\subsection{A dictionary between graph and group problems}
\label{GraphGroups}

The following are classical graph theoretic problems for which admit algebraic counterparts in the context of right-angled Artin groups. Note that suitable versions of the automorphism problem and of the decomposition problem, which are respectively solvable in quasi-polynomial time and in polynomial time, were described respectively in Section~\ref{AutProblem} and Section~\ref{DecProblem} above.

\vspace{.5cm}

\subsubsection{The clique problem} A \emph{k-clique} is a complete graph on $k$ vertices. Given a graph $\Gamma$, the clique problem consists in finding the induced subgraphs of $\Gamma$ which are $k$-cliques for a certain $k$. If $A(\Gamma)$ is the associated right-angled Artin group, this corresponds to finding the subsets of the
 Artin generators which give rise to free abelian subgroups of rank $k$. There are many different variations on this problem, for instance: find a maximal (with respect to inclusion) $k$-clique inside $\Gamma$; list all $k$-cliques; test if, given $k$, there exists a $j$-clique in $\Gamma$ for some $j>k$. The first of these problems is generally fixed-parameter intractable, the second is solvable in exponential time, and the third is NP-hard. Many algorithms dealing with particular cases have been proposed (see~\cite{APR99}).

\vspace{.1cm}

 \subsubsection{The independent set problem} An \emph{independent set} inside a graph $\Gamma$ is a set of vertices of $\Gamma$ such that there is no edge in the subgraph of $\Gamma$ which they span. The independent set problem consists in finding a maximal (with respect to inclusion) independent set inside the graph $\Gamma$. This problem is known to be NP-complete. Its counterpart in right-angled Artin group theory is finding a maximal subset of the set of Artin generators that generate a free group. As in the case of the clique problem (which in a rough sense is dual to this one), there are many natural variations on this problem. Again there is an extensive literature on the topic, especially for the case of sparse graphs~\cite{Ro86}.

 \vspace{.1cm}

 \subsubsection{Induced graphs} The previous questions belong to a large family of problems concerning induced graphs by subsets the set of vertices of a graph. In the context of right-angled Artin groups, such questions correspond to questions about the subgroups generated by subsets of the set of Artin generators.
 One such interesting problem is the \emph{induced subgraph isomorphism problem}, which given a pair of graphs $\Gamma$ and $\Gamma'$, consists of finding a subgraph of $\Gamma$ which is isomorphic to $\Gamma'$. In terms of right-angled Artin groups, this is equivalent to determining whether $A(\Gamma')<A(\Gamma)$, where this inclusion is of \emph{standard subgroups}, i.e. ones generated by subgraphs of the defining graph of $\Gamma$. Observe that in the special case $\Gamma'$ is a clique, we obtain the clique problem, whence we conclude that this problem is in general NP-complete. In the case where $\Gamma'$ has no edges, we obtain the independent set problem. Another interesting special case is the \emph{snake-in-the-box problem}, in which $\Gamma'$ is a chain (i.e a connected
 graph with two vertices of degree one and the remaining vertices of degree two) and in which $\Gamma$ is an hypercube (see~\cite{KOSU}).

 \vspace{.1cm}

 \subsubsection{Subdivision problems} Recall that given a graph $\Gamma$, a subdivision of $\Gamma$ is a graph obtained performing successive subdivisions in the edges. The problem in this case is, given graphs $\Gamma$ and $\Gamma'$, to find $\Gamma'$ as a subgraph of a subdivision of $\Gamma$. In the context of
 right-angled Artin groups, to subdivide edges of $A(\Gamma)$ is to add generators to the original set of generators $S$ in such a way that the generator added in every step
 commutes with exactly two of the existent commuting generators $v$ and $w$ in the previous step, and with the commutativity relation between $v$ and $w$ canceled. In these terms, the subdivision problem for right-angled Artin groups can be stated as follows: given
 such two groups with vertex generators $S$ and $S'$, decide if the second group is a subgroup of the group obtained adding generators and relations to $S$ by the previous process. Classical results such as Kuratowski's Theorem can also be formulated in this context. See \cite{BF}.

  \vspace{.1cm}

 \subsubsection{Graph coloring} A \emph{(vertex) coloring} of a graph is an assignment of colors to the vertices of a graph in such a way that no two adjacent vertices are assigned the same color. The \emph{chromatic number} of a graph is the minimum number of colors that are necessary to color the graph.  Other variations include the dual \emph{edge coloring}, an assignment of colors to the edges of a graph in such a way that no two edges incident to a common vertex are assigned the same color. The chromatic number for edges is defined as in the previous case, and in the context of right-angled Artin groups, can be interpreted in the following way. Let $F$ be a free group on a set $S$ of generators, $C$ a set of commutators of the elements of $S$, and $G=F/\langle\langle C\rangle\rangle$ the corresponding right-angled Artin group. Now if $C_1\ldots,C_m$ is a partition of $C$, the chromatic number associated to $G$ is the minimum cardinal of a partition such that it does not contain two commutators of the form $[s_0,s_1]$ and $[s_0,s_2]$. The edge coloring problem is known to be NP-complete, even in the case where one wants to decide if a graph is colorable with at most three colors~\cite{Hol81}. The papers~\cite{KK13} and~\cite{KK14} by Kim and the third author investigate the relationship between chromatic numbers and right-angled Artin groups.

  \vspace{.1cm}

 \subsubsection{Vertex cover problem} Given a graph $\Gamma$, a \emph{vertex cover} is a subset $V'$ of the set of vertices $V$ such that every edge of $\Gamma$ is incident with at least one vertex of $V'$. Given a certain $k>0$, the \emph{vertex cover problem} is the problem of deciding if there exists a vertex cover of $\Gamma$ with exactly $k$ vertices. This problem is known to be NP-complete, even for planar graphs, although it is fixed-parameter tractable. A natural variation is the minimum vertex cover problem, which consists of finding the minimum $k$ for which there exists a vertex cover with $k$ vertices. This problem is known to be NP-hard. In the context of right-angled Artin groups, the vertex cover problem for a fixed $k$ is equivalent to deciding whether there exists a subset $S'\subseteq S$ of a set of vertex generators consisting of $k$ generators such that any generator in $S\backslash S'$ commutes with at least one of the generators in $S'$. The minimum vertex cover problem has an analogous interpretation \cite{GJ}.

  \vspace{.1cm}

 \subsubsection{Arboricity} Recall that a \emph{forest} is an acyclic graph, i.e. a graph whose connected components are trees.  The  \emph{arboricity} of a graph $\Gamma$ is defined as the minimum number $k$ such that there exists $k$ subgraphs of $\Gamma$ that are forests and whose union contains all the edges of $\Gamma$. The problem of arboricity is to find such a $k$, and can be solved in polynomial time. For right-angled Artin groups, given a set of vertex generators $S$, the arboricity can translated in the following way. Given a natural number $k$, consider a collection $C$ of non-empty subsets $S_i\subseteq S$, for $1\leq i\leq k$, such that $S_m\neq S_n$ for $1\leq m<n\leq k$ and $\bigcup_{1\leq i\leq k}S_i=S$. Assume that the following condition does not hold for any $S_i$: there exists a subset $\{x_0,\ldots ,x_{n_i}\}\subseteq S_i$, $n_i\geq 3,$ such that the commutators $[x_l,x_{l+1}]$ are trivial for every $0\leq l\leq n_i$, where the subscripts are taken modulo $n_i+1$. Equivalently, the vertex generators in $S_i$ do not generate a right-angled Artin group on a cycle. Then the smallest $k$ for which such a collection $C$ exists will be the \emph{arboricity} of the right-angled Artin group. See \cite{GW}.

 \subsection{Open problems}

We close with some open problems in cryptography:
\begin{enumerate}
\item What other graph theoretic problems can be translated into right-angled Artin group theory in a way which yields interesting complexity results and the possibility for new platforms for cryptosystems?

\item Can one find a secret sharing scheme using the decomposition problem in right-angled Artin groups which transmits over a public channel?

\item In \cite{KK2003}, a secret sharing scheme has been proposed using graph coloring. Can the platform be modified to use right-angled Artin groups?
\item In \cite{K98}, the Polly Cracker public key cryptosystem is proposed. It is based on graph 3-coloring problem, which is known to be an NP-hard problem. Can this cryptosystem be modified to use right-angled Artin groups as a platform?
\end{enumerate}

\section*{Acknowledgements}

The authors thank S. Kim and A. Sale for helpful comments and corrections. The authors are indebted to an anonymous referee who read
the manuscript very carefully and provided a large number of helpful comments and corrections which greatly improved the paper.

\vspace{.5cm}

Ram\'{o}n Flores is supported by FEDER-MEC grant MTM2016-76453-C2-1-P. Delaram Kahrobaei is partially supported by a PSC-CUNY grant from the CUNY Research Foundation, the City Tech Foundation, and ONR (Office of Naval Research) grant N00014-15-1-2164. Thomas Koberda is partially supported by an Alfred P. Sloan Foundation Research Fellowship and by NSF Grant DMS-1711488. We thank International Center for Mathematical Sciences (ICMS), Edinburgh, which made this collaboration possible as well as the NSF grant which supported DK's and TK's travels. We acknowledge the Institut Henri Poincar\'e (IHP) for the NSF grant DMS-1700168 travel grant for DK, during the program on Analysis for Quantum Information Theory.

\end{document}